\documentclass[final,1p,times]{elsarticle}

\usepackage{amssymb}
\usepackage{amsthm}
\usepackage{mathrsfs}

\usepackage{mathptmx}
\usepackage{amsmath,amsfonts}
\usepackage{url,hyperref,multirow}
\usepackage{float}
\usepackage{color}
\usepackage{epsfig,bm,caption2,epstopdf}

\usepackage{setspace}
\usepackage{exscale}
\usepackage{relsize}
\usepackage{booktabs}
\usepackage{ifpdf}

\usepackage{dsfont}

\newtheorem*{notation}{Notation}
\newtheorem{theorem}{Theorem}[section]
\newtheorem{remark}{Remark}[section]

\newtheorem{lemma}{Lemma}[section]
\newtheorem{example}{Example}[section]

\journal{}

\begin{document}
	
	\begin{frontmatter}
		
		
		
		\title{An efficient numerical scheme for two-dimensional nonlinear time fractional Schr\"{o}dinger equation}
		
		
		
		\author[OUC]{Jun Ma }
		\ead{19819758528@163.com}
		
		\author[SLUAF]{Tao Sun}
		\ead{sunt@lixin.edu.cn}
		
		\author[OUC]{Hu Chen \corref{chen}}
		\ead{chenhu@ouc.edu.cn} \cortext[chen]{Corresponding author.}

		
		%
		

		\address[OUC]{School of Mathematical Sciences, Ocean University of China, Qingdao 266100, China}
		\address[SLUAF]{School of Statistics and Mathematics, Shanghai Lixin University of Accounting and Finance, Shanghai 201209, China}
		\begin{abstract}
			In this paper, a linearized fully discrete scheme is proposed to solve the two-dimensional nonlinear time fractional Schr\"{o}dinger equation with weakly singular solutions, which is constructed by using L1 scheme for Caputo fractional derivative, backward formula for the approximation of nonlinear term and five-point difference scheme in space. We rigorously prove the unconditional stability and pointwise-in-time convergence of the fully discrete scheme, which does not require any restriction on the grid ratio. Numerical results are presented to verify the accuracy of the theoretical analysis.
		\end{abstract}
		
		\begin{keyword}
			
			Nonlinear time fractional Schr\"{o}dinger equation\sep  L1 scheme\sep  Local error estimate \sep unconditional stability
			
			\MSC[2020] Primary 65M15, 65M06
		\end{keyword}
		
	\end{frontmatter}
	
	

	\section{Introduction}
	In recent years, due to the widespread applications and profound influences in physical fields such as quantum mechanics, fractional dynamics, see \cite{Tofighi,Iomin1,Iomin,Laskin,Lenzi,Naber}, the time fractional Schr\"{o}dinger equation has become an indispensable tool in scientific research and continues to attract numerous scholars for in-depth exploration and extensive attention. \\
	\indent Consider the following two-dimensional nonlinear time fractional Schr\"{o}dinger equation (NTFS) with corresponding boundary and initial conditions:
	\begin{align}
		&\text{i}D_t^\alpha u+\Delta u+f(|u|^2)u=0,&&(x,y)\in\Omega, 0<t\leq T,\label{eqution1}\\&u(x,y,0)=u_0(x,y),&&(x,y)\in\Omega,\label{eqution2}\\&u(x,y,t)=0,&&(x,y)\in\partial\Omega, 0\leq t\leq T,\label{eqution3}
	\end{align}
	where $u = u(x,y,t)$ is the complex-valued function, $0<\alpha<1,$ the imaginary unit $\text{i}^2=-1$, the domain $\Omega:=(0,L)^2$, $u_0\in C(\bar{\Omega})$ and $f \in C^1(\mathbb{R}^+)$ is a given real-valued function. In \eqref{eqution1}, $D_t^\alpha$ represents standard Caputo fractional derivative operator (see \cite{Diethelm}), which is defined by
	\begin{align}
		D_t^\alpha u(x,y,t)=\frac1{\Gamma(1-\alpha)}\int_0^t(t-\xi)^{-\alpha}\frac{\partial u(x,y,\xi)}{\partial \xi}\mathrm{d}\xi,\quad t>0.
	\end{align}
	\indent As $\alpha$ approaches 1, the fractional derivative $D_t^\alpha u$ converges to the first-order derivative $\partial u/\partial t$, causing \eqref{eqution1} reducing to the integer-order nonlinear Schr\"{o}dinger equation (NLS). For numerical methods on solving NLS, one can refer to \cite{Antoine,Chang,Henning,LiXiangGui,WangTingchun}. Recently, there have been research methods for the analytical and numerical solutions of the  time fractional Schr\"{o}dinger equation (NTFS). For example, Li et al. \cite{Li} proposed a linearized numerical scheme incorporating the L1 method for approximating the Caputo time-fractional derivative in
	temporal direction and the Galerkin finite element method in spatial discretization, and they obtained the optimal error estimate by using time-space error splitting argument. Liu et al. \cite{Liu} introduced the linear Alikhanov finite difference scheme for solving nonlinear time fractional Schr\"{o}dinger equation. Wang et al. \cite{Wang} developed linearized L2-$1_\sigma$ and FL2-$1_\sigma$ formulas, and proved the unconditional stability and optimal error estimates of these two numerical schemes for solving NTFS. Chen et al. \cite{ChenXL} proposed two linearized compact ADI methods to solve the two-dimensional nonlinear time fractional Schr\"{o}dinger equation. For more numerical methods to solve NTFS, one can refer to \cite{LiMeng,Mohebbi,Wei}. All above works were obtained by assuming that the solutions are sufficiently smooth, which are too restrictive due to the singularity of the initial layer are typical for time fractional derivative problems. More details can be seen in \cite{Kopteva,Martin}. Then, taking into account the initial singularity, some researchers conducted numerical studies on time fractional Schr\"{o}dinger equations. Zhang et al. \cite{Zhangchen20} conducted an error analysis of a fully discrete scheme for linearized time fractional Schr\"{o}dinger equation with initial singularity.
	Hu et al. \cite{Hu} developed a two-grid finite element method with nonuniform L1 scheme for solving the nonlinear time-fractional Schr\"{o}dinger equation. Qin et al. \cite{Qin} proposed Alikhanov linearized Galerkin methods for solving the NTFS with non-smooth solutions, and derived unconditionally optimal estimates for the fully discrete scheme using fractional time-space splitting. In \cite{Yuan}, Yuan et al. presented a linearized transformed L1 Galerkin finite element method and proved unconditionally optimal error estimates through a combination of new discrete fractional Gr\"{o}nwall inequalities and Sobolev embedding theorems.  In \cite{Yuanli2023}, the authors proposed a high order fast time-stepping method for the time-space nonlinear fractional {S}chr\"{o}dinger equations. In our previous work \cite{MaChen24}, we established local error estimate of L1 scheme on graded mesh for linearized time fractional chr\"{o}dinger equation.\\
	\indent However, to our knowledge, there has been no research on the local convergence problem of NTFS with non-smooth solutions. Therefore, we will establish local error estimate without any restriction on the grid ratio for the NTFS under weaker assumptions of regularity. In this paper, we construct a linearized L1 finite difference scheme for NTFS. Then, under weaker conditions $f \in C^1(\mathbb{R}^+)$ (specifically, $f \in C^1(\mathbb{R}^+)$ compared to $f \in C^3(\mathbb{R}^+)$ in \cite{Li,Qin,Wang}, $f \in C^2(\mathbb{R})$ in \cite{ChenXL,Hu}, and $f \in C^1(\mathbb{R}^+)$ in \cite{Liu,Yuanli2023}), and combining with inverse Sobolev inequalities and a discrete fractional Gr\"{o}nwall-type inequality, the local error estimate of the proposed scheme is established without any restriction on the grid ratio.\\
	\indent The rest of this paper is structured as follows.
	In Section \ref{section2}, a linearized fully discrete numerical scheme is proposed for problem \eqref{eqution1} and some preliminary lemmas used to prove error estimates are stated.
	In Section \ref{section3}, the main results of convergence and  stability analysis of our scheme are provided. We use several numerical tests to illustrate the accuracy of our theoretical analysis in Section \ref{section4}. The final part ends this paper with a brief conclusion.
	\begin{notation}
		Throughout the paper, $C$ is a generic constant, and the value of $C$ may change from line to line, but is
		always independent of the mesh sizes.		
	\end{notation}

	\section{The linearized fully discrete method and preliminary lemmas}\label{section2}
	For a positive integer $N$, let the time step size $\tau=T/N$, and denote grid points $t_n=n\tau$ for $n=0,1,\dots,N$.
	Grid points of the spatial mesh are given by $x_j=jh, y_k=kh, 0\leq j, k\leq M$, where $h=L/M$ and $M$ is a positive integer. We use $u_{j,k}^n$ and $U_{j,k}^n$ to represent the exact value and nodal approximation of the solution $u(x_j,y_k,t_n)$ for all $j,k$ and $n$.\\
	\indent Then we introduce the well-known L1 scheme to approximate the Caputo fractional derivative $D_t^\alpha u(x,y,t)$ of $u$ with order $\alpha$:
	\begin{equation}\label{eqution4}
		\begin{aligned}
			D_t^\alpha u(x_j,y_k,t_n)&=
			\frac{1}{\Gamma(1-\alpha)}\int_{0}^{t_{n}}\frac{u'(x_j,y_k,s)}{(t_{n}-s)^{\alpha}}ds\\ &\approx\frac{1}{\Gamma(1-\alpha)}\sum_{i=0}^{n-1}\int_{t_i}^{t_{i+1}}\frac{u_{j,k}^{i+1}-u_{j,k}^i}{\tau}\frac{1}{(t_n-s)^\alpha}ds \\ &=\frac1\mu\Big(a_0u_{j,k}^n-\sum_{i=1}^{n-1}(a_{n-i-1}-a_{n-i})u_{j,k}^i-a_{n-1}u_{j,k}^0\Big)\\
			&:=D_\tau^\alpha u_{j,k}^n,
		\end{aligned}
	\end{equation}
	where $$\mu=\tau^{\alpha}\Gamma(2-\alpha),\quad a_{i}=(i+1)^{1-\alpha}-i^{1-\alpha}, \quad i\geq0.$$
	
	Then define the positive multipliers $\theta_n$ (see \cite[Eq.(11)]{Chen}) by
	$$\theta_0=\left( \frac1\mu a_0\right) ^{-1},\quad\theta_n=a_0^{-1}\sum_{i=1}^n(a_{i-1}-a_i)\theta_{n-i},\quad1\leq n\leq N-1.$$
	
	\begin{lemma} \label{lemma2.1}
		\cite[Corollary 16]{Chen} The sequence $\theta_n$ is monotonically decreasing and satisfies
		\begin{align*}
			\theta_n\leq\Gamma(2-\alpha)\tau^\alpha(n+1)^{\alpha-1}.
		\end{align*}
	\end{lemma}
	\begin{lemma} \label{lemma2.2}
			\cite[Lemma 2.1]{Liao} For $n = 1,\dots , N$, one has
			\begin{align*}
				\sum_{i=1}^n\theta_{n-i}\leq\frac{t_n^\alpha}{\Gamma(1+\alpha)}.
			\end{align*}
			\begin{proof}
				It can be easily proved by taking $v(t)=t^\alpha$ in    \cite[Lemma 2.1, Eq. (2.8) ]{Liao}.
			\end{proof}
	\end{lemma}
	Let $\beta\geq0$ be constant. Define
	\begin{align*}
		K_{\beta,n}:=\begin{cases}1+\frac{1-n^{1-\beta}}{\beta-1}&\mathrm{if~}\beta\neq1,\\1+\ln n&\mathrm{if~}\beta=1.\end{cases}
	\end{align*}
	\begin{lemma} \label{lemma2.3}
		\cite[Lemma 5.10]{Chen1} For $1\leq n \leq N$, and a constant $\beta\geq0$, one has
		\begin{align*}
			\sum_{i=1}^ni^{-\beta}\theta_{n-i}\leq\tau^\alpha K_{\beta,n}\left(\frac n2\right)^{\alpha-1}+\frac{2^{\beta-\alpha}}\alpha\tau^\alpha n^{-\beta+\alpha}.
		\end{align*}
	\end{lemma}
	Let $\{u^i\}_{i=0}^{N}$ be an arbitrary mesh function. We define an operator $E_t^\alpha$ as
	\begin{align}\label{eqution13}
		E_t^\alpha(u^0)=0,\quad E_t^\alpha(u^n)=\sum_{i=1}^n\theta_{n-i}u^i\quad\mathrm{for~}n=1,\ldots,N.
	\end{align}
	\begin{lemma} \label{lemma2.4}
		\cite[Lemma 2]{Cao} For arbitrary mesh function $\{u^i\}_{i=0}^{N}$, one has
		\begin{align*}
			E_t^\alpha(D_\tau^\alpha u^n)=u^n-u^0\quad\mathrm{for~}n=1,\ldots,N.
		\end{align*}
	\end{lemma}
	In this paper, we assume that the exact solution $u$ satisfies 
	\begin{align}\label{eqution8}
		\left|\frac{\partial^k u(x,y,t)}{\partial x^i\partial y^j }\right|\leq C, ~~ 0\leq i+j=k\leq 4,\quad
		\left|\frac{\partial^l u(x,y,t)}{\partial t^l}\right|\leq C(1+t^{\alpha-l}),\quad l=0,1,2,\quad (x,y)\in\bar{\Omega},~ 0<t\leq T.		
	\end{align}
\begin{remark}
	The time regularity of the solution  in \eqref{eqution8} is a realistic assumption for time-fractional derivative problems, \cite[Theorem 3.1]{Jin} proved such regularity in the case $l=1$ for nonlinear subdiffusion equations, and similar results can be found in \cite[Theorem 2.2]{Du} for  time-fractional Allen-Cahn equations.
\end{remark}
	Denote truncation error $r^n=D_\tau^\alpha u(x,y,t_n)-D_t^\alpha u(x,y,t_n)$, then from \cite[Lemma 1]{Gracia}, we can get the following lemma:
	\begin{lemma} \label{lemma2.5}
		Assume that $u$ satisfies the assumption of the regularity \eqref{eqution8}, one has
		\begin{align*}
			|r^n|\leq Cn^{-\min\{\alpha+1,2-\alpha\}},\quad1\leq n\leq N.
		\end{align*}
	\end{lemma}
	\begin{lemma} \label{lemma2.6}
		\cite[Corollary 5.11]{Chen1} For $1\leq n \leq N$, one has
		\begin{align*}
			\sum_{i=1}^n|r^i|\theta_{n-i}\leq C\tau t_n^{\alpha-1}.
		\end{align*}
	\end{lemma}	
	For simplicity, we define a space of grid functions as $W_h:=\left\{u_{j,k}|u_{j,k}\in\mathbb{C} ~\text{for}~ 0\leq j,k \leq M\right\}$ and $W_h^0:=\left\{ u_{j,k}|u_{j,k}\in W_h, u_{j,k}=0 \text{ if} j=0,M \text{ or } k=0,M\right\}$. As usual, for any grid function
	$u^n \in W_h^0$, we introduce the following notations,
	\begin{align*}
		&\delta_{x}^{2}u_{j,k}^{n}=\frac{u_{j-1,k}^{n}-2u_{j,k}^{n}+u_{j+1,k}^{n}}{h^{2}},\quad
		\delta_{y}^{2}u_{j,k}^{n}=\frac{u_{j,k-1}^{n}-2u_{j,k}^{n}+u_{j,k+1}^{n}}{h^{2}},\\
		&\Delta_hu_{j,k}^n=\delta_{x}^{2}u_{j,k}^{n}+\delta_{y}^{2}u_{j,k}^{n}.
	\end{align*}
	For any grid functions $u, v\in W_h^0$, we define discrete inner products and discrete norms as
	\begin{align*}
		& \left( u,v\right) :=h^2\sum_{j=1}^{M-1}\sum_{k=1}^{M-1}u_{j,k}\overline{v}_{j,k},\quad\|u\|:=\left(  u,u\right) ^{\frac{1}{2}}, \\
		&|u|_2:=\left( \Delta_{h}u,\Delta_{h}u\right)^{\frac{1}{2}},\quad \|u\|_\infty:=\max_{1\leq j,k\leq M}|u_{j,k}|,
	\end{align*}
	where $\overline{v}$ is the complex conjugate of $v$.\\
	\indent A standard five point difference approximation is used to discretize the Laplacian term $\Delta u_{j,k}^n$,
	\begin{align*}
		\Delta u_{j,k}^n=\Delta_h u_{j,k}^n+O(h^2).
	\end{align*}
	Let $U_{j,k}^n$ be the approximations of $u(x_j,y_k,t_n)$, then the problem \eqref{eqution1}--\eqref{eqution3} is now approximated by the following fully discrete problem:
	\begin{align}
		&\text{i}D_\tau^\alpha U_{j,k}^n+\Delta_h U_{j,k}^n+f(|U_{j,k}^{n-1}|^2)U_{j,k}^{n-1}=0,&& 1\leq j,k\leq M-1, 1\leq n\leq N,\label{eqution5}\\&U_{j,k}^{0}=u_0(x_j,y_k),&&1\leq j,k\leq M-1,\label{eqution6}\\&U_{j,k}^{n}=0,&&j=0,M~ or ~k=0,M, 0\leq n\leq N.\label{eqution7}
	\end{align}
	
	\section{Stability and convergence analysis}\label{section3}
	In this section, we aim to obtain the stability and local error estimate. Firstly, We introduce a discrete Gr\"{o}nwall inequality and other lemmas.
	\begin{lemma} \label{lemma3.1}
		\cite[Lemma 5.1]{Mariam} Let $\{u
		\}_{n=1}^N$ be a non-negative sequence, assume that there exists constant $\eta_1, \eta_2\in [0,1)$, and $b_1,b_2,b\geq0$, such that
		\begin{align*}
			u^n\leq b_1t_n^{-\eta_1}+b_2t_n^{-\eta_2}+b\tau\sum_{i=1}^{n-1}t_{n-i}^{\alpha-1}u^i\quad\mathrm{for~}n=1,\ldots,N.
		\end{align*}
		Then there exists a constant $C = C(\eta_1, \eta_2, \alpha, b,t_N )$ such that
		\begin{align*}
			u^n\leq C(b_1t_n^{-\eta_1}+b_2t_n^{-\eta_2}),\quad1\leq n\leq N.
		\end{align*}
	\end{lemma}	
	\begin{remark} \label{remark1}
		Let $\frac{1}{2}<\alpha<1$ in Lemma \ref{lemma3.1}, we can find  $C$ is $\alpha$-robust because the bound does not blow up when $\alpha$ approaches $1^-$.
		\begin{align*}
			u^{n}& \leq b_1t_n^{-\eta_1}+b_2t_n^{-\eta_2}+b\tau\sum_{i=1}^{n-1}t_{n-i}^{\alpha-1}\left(b_1t_i^{-\eta_1}+b_2t_i^{-\eta_2}+b\tau\sum_{j=1}^{i-1}t_{i-j}^{\alpha-1}u^j\right) \\
			&=b_1t_n^{-\eta_1}+b_2t_n^{-\eta_2}+bb_1\left(\tau\sum_{i=1}^{n-1}t_{n-i}^{\alpha-1}t_i^{-\eta_1}\right)+bb_2\left(\tau\sum_{i=1}^{n-1}t_{n-i}^{\alpha-1}t_i^{-\eta_2}\right) \\
			&+b^2\tau^2\sum_{i=1}^{n-1}\sum_{j=1}^{i-1}t_{n-i}^{\alpha-1}t_{i-j}^{\alpha-1}u^j.
		\end{align*}
		Using the inequality
		\begin{align*}
			\tau\sum_{i=j+1}^{n-1}t_{n-i}^{\alpha-1}t_{i-j}^{\beta-1}\leq B(\alpha,\beta)t_{n-j}^{\alpha+\beta-1},\quad 0<\alpha<1,~\beta>0
		\end{align*}
		which follows from \cite[Lemma 6.1]{Dixon86}, where $B(\alpha,\beta)$ is the Beta function, and reversing the order of summation, we get
		\begin{align*}
			u^{n}& \leq b_1t_n^{-\eta_1}+b_2t_n^{-\eta_2}+bb_1B(\alpha,1-\eta_1)t_n^{\alpha-\eta_1}+bb_2B(\alpha,1-\eta_2)t_n^{\alpha-\eta_2}+b^2\tau\sum_{j=1}^{n-2}u^j \left(\tau\sum_{i=j+1}^{n-1}t_{n-i}^{\alpha-1}t_{i-j}^{\alpha-1}\right)\\
			&\leq \left( b_1+bb_1B(\alpha,1-\eta_1)t_n^\alpha \right) t_n^{-\eta_1}+\left( b_2+bb_2B(\alpha,1-\eta_2)t_n^\alpha \right) t_n^{-\eta_2}+b^2B(\alpha,\alpha)\tau\sum_{j=1}^{n-2}t_{n-j}^{2\alpha-1}u^j\\
			&\leq \left( b_1+bb_1B(\alpha,1-\eta_1)t_n^\alpha \right) t_n^{-\eta_1}+\left( b_2+bb_2B(\alpha,1-\eta_2)t_n^\alpha \right) t_n^{-\eta_2}+b^2B(\alpha,\alpha)\tau t_{N}^{2\alpha-1}\sum_{j=1}^{n-2}u^j.			
		\end{align*}
		If $\eta_1 \geq \eta_2$, we set $z^n = t_n^{\eta_1}u^n$ so that
		\begin{align*}
			z^{n}& \leq \left( b_1+bb_1B(\alpha,1-\eta_1)t_n^\alpha \right)+\left( b_2+bb_2B(\alpha,1-\eta_2)t_n^\alpha \right) t_n^{\eta_1-\eta_2}+b^2B(\alpha,\alpha)\tau t_{N}^{2\alpha-1}t_n^{\eta_1}\sum_{j=1}^{n-2}t_j^{-\eta_1}z^j.			
		\end{align*}
		Finally, an application of the standard discrete Gr\"{o}nwall's inequality \cite[Lemma 1.4.2]{Quarteroni} implies
		\begin{align*}
			z^{n}& \leq e^{\frac{1}{1-\eta_1}b^2B(\alpha,\alpha)t_{N}^{2\alpha-1}t_n^{\eta_1}t_{n-1}^{1-\eta_1}}\left[  \left( b_1+bb_1B(\alpha,1-\eta_1)t_n^\alpha \right)+ \left( b_2+bb_2B(\alpha,1-\eta_2)t_n^\alpha \right) t_n^{\eta_1-\eta_2}\right]\\
			& \leq e^{\frac{1}{1-\eta_1}b^2B(\alpha,\alpha)t_{N}^{2\alpha}}\left[  \left( b_1+bb_1B(\alpha,1-\eta_1)t_n^\alpha \right)+ \left( b_2+bb_2B(\alpha,1-\eta_2)t_n^\alpha \right) t_n^{\eta_1-\eta_2}\right],
		\end{align*}
		then
		\begin{align*}
			u^{n}& \leq e^{\frac{1}{1-\eta_1}b^2B(\alpha,\alpha)t_{N}^{2\alpha}}\left[  \left( b_1+bb_1B(\alpha,1-\eta_1)t_n^\alpha \right)t_n^{-\eta_1}+ \left( b_2+bb_2B(\alpha,1-\eta_2)t_n^\alpha \right) t_n^{-\eta_2}\right].
		\end{align*}
		Thus, we have $$C=e^{\frac{1}{1-\eta_1}b^2B(\alpha,\alpha)t_{N}^{2\alpha}}\max\{1+bB(\alpha,1-\eta_1)t_n^\alpha,1+bB(\alpha,1-\eta_2)t_n^\alpha\}.$$
		If $\eta_1 < \eta_2$, we set $z^n = t_n^{\eta_2}u^n$, using the same method, we can obtain $$C=e^{\frac{1}{1-\eta_2}b^2B(\alpha,\alpha)t_{N}^{2\alpha}}\max\{1+bB(\alpha,1-\eta_1)t_n^\alpha,1+bB(\alpha,1-\eta_2)t_n^\alpha\}.$$
		Therefore, the constant $C$ is $\alpha$-robust.
	\end{remark}	
	\begin{lemma} \label{lemma3.2}
		\cite[p.13]{Zhou} For arbitrary grid function $u \in W_h^0$, one has
		\begin{align*}
			\|u\|_\infty\leq C\|u\|^{\frac12}(\|u\|+|u|_2)^{\frac12}.
		\end{align*}
	\end{lemma}	
	\begin{lemma} \label{lemma3.3}
		\cite[Lemma 3.3]{Liu} For arbitrary mesh function $u \in W_h^0$, there is
		\begin{align*}
			\|u\|_\infty\leq Ch^{-1}\|u\|.
		\end{align*}	
	\end{lemma}	
	\textcolor{blue}{\begin{lemma} \label{lemma3.4}
		For any grid function $u \in W_h$, one has
		\begin{align*}
			\left\| D_\tau^\alpha u^n\right\| \leq \dfrac{2}{\tau^{\alpha}\Gamma(2-\alpha)}\max_{0\leq s\leq n}\|u^s\|,\quad n=1,2,\cdots,N.
		\end{align*}	
		\begin{proof}
			Firstly, it is easy to obtain $1=a_0> a_1>\cdots> a_N>0$, then we have
			\begin{align*}
				\|D_\tau^\alpha u^n\|& = \frac{1}{\mu}\left\| a_0u^{n}-\sum_{s=1}^{n-1}(a_{n-s-1}-a_{n-s})u^s-a_{n-1}u^0\right\| \\
				&\leq\frac{1}{\mu}\left( a_0\left\| u^{n}\right\| +\sum_{s=1}^{n-1}(a_{n-s-1}-a_{n-s})\left\| u^s\right\| +a_{n-1}\left\| u^0\right\| \right) \\
				&\leq\frac{1}{\mu}\left( a_0 +\sum_{s=1}^{n-1}(a_{n-s-1}-a_{n-s}) +a_{n-1} \right) \max_{0\leq s\leq n}\|u^s\|\\
				&=\frac{2a_0}{\mu}\max_{0\leq s\leq n}\|u^s\|= \dfrac{2}{\tau^{\alpha}\Gamma(2-\alpha)}\max_{0\leq s\leq n}\|u^s\|.
			\end{align*}		
			This proof is completed.
		\end{proof}
	\end{lemma}} 
	The stability of scheme \eqref{eqution5}--\eqref{eqution7} is demonstrated in the following theorem. Assume that $\widetilde{U}_{j,k}^n$ is the solution of the following scheme with initial value $\widetilde{u}_0$:
	\begin{align}
		&\text{i}D_\tau^\alpha \widetilde{U}_{j,k}^n+\Delta_h \widetilde{U}_{j,k}^n+f(|\widetilde{U}_{j,k}^{n-1}|^2)\widetilde{U}_{j,k}^{n-1}=0,&&(x_j,y_k)\in\Omega_h, 1<n\leq N,\label{eqution15a}\\&\widetilde{U}_{j,k}^{0}=\widetilde{u}_0(x_j,y_k),&&(x_j,y_k)\in\Omega_h,\label{eqution15b}\\&\widetilde{U}_{j,k}^{n}=0,&&(x_j,y_k)\in\partial\Omega_h, 0\leq n\leq N.\label{eqution15c}
	\end{align}
	\begin{theorem}\label{theorem 3.1}
		The fully discrete scheme \eqref{eqution5}--\eqref{eqution7} is unconditionally stable with respect to initial value. Suppose $U_{j,k}^n,\widetilde{U}_{j,k}^n$ be the solutions of schemes \eqref{eqution5}--\eqref{eqution7},  \eqref{eqution15a}--\eqref{eqution15c} respectively, then for $n = 1, \dots , N$, one has
		\begin{align}\label{eqution17}
			\left\| U^n-\widetilde{U}^n\right\|\leq C\left\| u_0-\widetilde{u}_0\right\|.
		\end{align}			
	\end{theorem}
	\begin{proof}
		Let $\widetilde{e}_{j,k}^n=U_{j,k}^n-\widetilde{U}_{j,k}^n$, then $\widetilde{e}_{j,k}^0=u_0(x_j,y_k)-\widetilde{u}_0(x_j,y_k)$, subtracting the discrete scheme \eqref{eqution15a} from \eqref{eqution5} gives the following equation:
		\begin{align}\label{eqution16}
			\text{i}D_\tau^\alpha \widetilde{e}_{j,k}^n+\Delta_h \widetilde{e}_{j,k}^n+\widetilde{R}_{j,k}^n=0,
		\end{align}
		where $\widetilde{R}_{j,k}^n=f(|U_{j,k}^{n-1}|^2)U_{j,k}^{n-1}-f(|\widetilde{U}_{j,k}^{n-1}|^2)\widetilde{U}_{j,k}^{n-1}$.
		Using Lagrange mean value theorem, we can get
		\begin{align*}
			\left\|\widetilde{R}^n \right\| &=\left\|\left(f(|U^{n-1}|^2)-f(|\widetilde{U}^{n-1}|^2)\right)U^{n-1}+f(|\widetilde{U}^{n-1}|^2)\left(U^{n-1}-\widetilde{U}^{n-1}\right)\right\|\\
			&=\left\|f'(\xi)\left(|U^{n-1}|^2-|\widetilde{U}^{n-1}|^2\right)U^{n-1}+f(|\widetilde{U}^{n-1}|^2)\left(U^{n-1}-\widetilde{U}^{n-1}\right)\right\|,
		\end{align*}
		where $\xi$ is a constant between $|U^{n-1}|^2$
		and $|\widetilde{U}^{n-1}|^2$,
		then due to $\| U^{n-1}\|_\infty\leq C$, $\| \widetilde{U}^{n-1}\|_\infty\leq C$ (which will be shown in \eqref{eqution14} of Theorem \ref{theorem 3.2}), it is easily seen that $\xi \leq C$. Thus, one has
		\begin{equation}\label{eqution20}
			\begin{aligned}
				\left\|\widetilde{R}^n \right\|
				&\leq\left\|f'(\xi)\left(U^{n-1}\left( \overline{\widetilde{e}^{n-1}}\right) +\left( \widetilde{e}^{n-1}\right)\overline{\widetilde{U}^{n-1}}\right)U^{n-1}\right\|+	\left\|f(|\widetilde{U}^{n-1}|^2)\left(\widetilde{e}^{n-1}\right)\right\|\\
				&\leq C \left\|\widetilde{e}^{n-1}\right\|.
			\end{aligned}
		\end{equation}
		Taking the inner product of \eqref{eqution16} with $\widetilde{e}^n_{j,k}$ and taking the imaginary part of the result, we arrive at
		\begin{align*}
			\mathrm{Re}(D_\tau^\alpha \widetilde{e}^n,\widetilde{e}^n) = -\mathrm{Im}\left( \widetilde{R}^n,\widetilde{e}^n\right).
		\end{align*}
		Then using the definition \eqref{eqution4} and Cauchy-Schwartz inequality, we have
		\begin{align*}
			\frac{1}{\mu}a_0\left\| \widetilde{e}^n\right\| ^2-\frac{1}{\mu}\sum_{i=1}^{n-1}(a_{n-i-1}-a_{n-i})\left\| \widetilde{e}^i\right\|\left\| \widetilde{e}^n\right\|-\frac{1}{\mu}a_{n-1}\left\| \widetilde{e}^0\right\|\left\| \widetilde{e}^n\right\|\leq\left\| \widetilde{R}^n\right\|\left\| \widetilde{e}^n\right\|.
		\end{align*}
		Then applying \eqref{eqution20}, i.e. $\left\| \widetilde{R}^n\right\|\leq C\left\| \widetilde{e}^{n-1}\right\|$, thus
		\begin{align*}
			D^\alpha_\tau\left\| \widetilde{e}^n\right\|\leq\left\| \widetilde{R}^n\right\|\leq C\left\| \widetilde{e}^{n-1}\right\|.
		\end{align*}
		Then using the definition of operator $E_t^\alpha$ \eqref{eqution13} to both sides of the above inequality, one gets
		\begin{align*}
			\left\| \widetilde{e}^n\right\|&\leq \left\| \widetilde{e}^0\right\|+C\sum_{i=1}^n\theta_{n-i}\left\| \widetilde{e}^{i-1}\right\|\leq \left\| \widetilde{e}^0\right\|+C\sum_{i=0}^{n-1}\theta_{n-i-1}\left\| \widetilde{e}^{i}\right\|\\
			&\leq \left\| \widetilde{e}^0\right\|+C\theta_{n-1}\left\| \widetilde{e}^0\right\|+C\sum_{i=1}^{n-1}\theta_{n-i-1}\left\| \widetilde{e}^{i}\right\|.		
		\end{align*}
		Using Lemma \ref{lemma2.1}, it holds that
		\begin{align*}
			\left\| \widetilde{e}^n\right\|&\leq\left\| \widetilde{e}^0\right\|+C\Gamma(2-\alpha)\tau^\alpha n^{\alpha-1}\left\| \widetilde{e}^0\right\|+C\Gamma(2-\alpha)\tau^\alpha\sum_{i=1}^{n-1}(n-i)^{\alpha-1}\left\| \widetilde{e}^{i}\right\|\\
			&=(1+C\Gamma(2-\alpha)\tau^\alpha n^{\alpha-1})\left\| \widetilde{e}^0\right\|+C\Gamma(2-\alpha)\tau\sum_{i=1}^{n-1}t_{n-i}^{\alpha-1}\left\| \widetilde{e}^{i}\right\|\\
			&\leq(1+C\Gamma(2-\alpha)T^\alpha)\left\| \widetilde{e}^0\right\|+C\Gamma(2-\alpha)\tau\sum_{i=1}^{n-1}t_{n-i}^{\alpha-1}\left\| \widetilde{e}^{i}\right\|.
		\end{align*}
		This together with Lemma \ref{lemma3.1} gives
		\begin{align*}
			\left\| \widetilde{e}^n\right\|\leq(1+C\Gamma(2-\alpha)T^\alpha)\left\| \widetilde{e}^0\right\|\leq C\left\| \widetilde{e}^0\right\|.
		\end{align*}	
		This proof is completed.
	\end{proof}
	Next, we define and prove the truncation error, which will be used in convergence analysis.\\
	Clearly, $u_{j,k}^n$ satisfies the following equation,
	\begin{align}\label{eqution9}
		\text{i}D_\tau^\alpha u_{j,k}^n+\Delta_h u_{j,k}^n+f(|u_{j,k}^{n-1}|^2)u_{j,k}^{n-1}=P_{j,k}^{n},
	\end{align}
	for $n = 1,2,\cdots, N$, where
	\begin{align}
		P_{j,k}^{n}=\mathrm{i}\left(D_\tau^\alpha u_{j,k}^n-D_{t_n}^\alpha u\right)+\left(\Delta_h u_{j,k}^n-\Delta u_{j,k}^n\right)+\left(f(|u_{j,k}^{n-1}|^2)u_{j,k}^{n-1}-f(|u_{j,k}^n|^2)u_{j,k}^n\right).
	\end{align}
	Using Lagrange mean value theorem and the standard Taylor's expansion, we have
	\begin{align*}
		&\left\|f(|u^{n-1}|^2)u^{n-1}-f(|u^n|^2)u^n\right\|\\
		&=\left\|\left(f(|u^{n-1}|^2)-f(|u^n|^2)\right)u^{n-1}+f(|u^n|^2)\left( u^{n-1}-u^{n}\right)\right\|\\
		&=\left\|f'(\xi_1)\left(|u^{n-1}|^2-|u^n|^2\right)u^{n-1}+f(|u^n|^2)\left( u^{n-1}-u^{n}\right)\right\|\\
		&\leq\left\|f'(\xi_1)\left(u^{n-1}\left( \overline{u^{n-1}-u^{n}}\right) +\left( u^{n-1}-u^{n}\right)\overline{u^{n}}\right)u^{n-1}\right\|+	\left\|f(|u^n|^2)\left( u^{n-1}-u^{n}\right)\right\|\\
		&\leq C \left\|u^{n-1}-u^{n}\right\|\leq C \left\|\int_{t_{n-1}}^{t_n}u^{\prime}(s)ds\right\|\leq C\int_{t_{n-1}}^{t_n}s^{\alpha-1}ds\\
		&\leq \begin{cases}C\tau^\alpha,&n=1,\\ Ct_{n-1}^{\alpha-1}\tau,&2\leq n\leq N\end{cases}\\
		&\leq C\tau t_{n}^{\alpha-1}=C \tau^\alpha n^{\alpha-1},
	\end{align*}
	{\color{blue}where $f \in C^1(\mathbb{R}^+)$ is used, and $\xi_1$ is a constant between $|u^{n-1}|^2$ and $|u^n|^2$, then by assumption \eqref{eqution8}, it is easily seen that $\xi_1\leq C$.} Therefore, one has
	\begin{align}\label{eqution11}
		\left\|P^n\right\|\leq {\color{blue}C_u} \left( \left\| r^n\right\| +h^2+\tau^\alpha n^{\alpha-1}\right).
	\end{align}
	Next we prove the pointwise-in-time convergence of the fully discrete scheme \eqref{eqution5}--\eqref{eqution7}. For $n = 0,\dots,N$, let $e_{j,k}^n=u_{j,k}^n-U_{j,k}^n$, then $e_{j,k}^0=0$. Subtracting the fully discrete scheme \eqref{eqution5} from \eqref{eqution9} gives the following error equation:
	\begin{align}\label{eqution10}
		\text{i}D_\tau^\alpha e_{j,k}^n+\Delta_he_{j,k}^{n}+R_{j,k}^{n}=P_{j,k}^{n},
	\end{align}
	where
	\begin{align*}
		R_{j,k}^{n}=f(|u_{j,k}^{n-1}|^2)u_{j,k}^{n-1}-f(|U_{j,k}^{n-1}|^2)U_{j,k}^{n-1}.
	\end{align*}
	
	  \begin{theorem}\label{theorem 3.2}
		Suppose that the problem \eqref{eqution1}--\eqref{eqution3} has a unique solution $u =
		u(x, y, t)$ satisfying \eqref{eqution8}, then the fully discrete scheme \eqref{eqution5}--\eqref{eqution7} has a unique solution $U_{j,k}^n$. For small $\tau$ and $h$, we have
		\begin{align}\label{eqution14}
			\|u^n-U^n\|\leq {\color{blue}K}(\tau t_n^{\alpha-1}+h^2),\quad\|U^n\|_\infty\leq {\color{blue}\max_{1\leq n\leq N}\|u^n\|_\infty+1},\quad n=1,2,\cdots,N,
		\end{align}	
   {\color{blue} where $K$ is defined in \eqref{keyc}.}
	\end{theorem}
	\begin{proof}
		We use mathematical induction to prove the theorem. Firstly,
		when $n=1$, $R_{j,k}^1=0$, taking the inner product of \eqref{eqution10} with $e^1_{j,k}$
		\begin{align*}
			\text{i}\frac{1}{\mu}	\|e^1\|^2-\left(\nabla_h e^1,\nabla_h e^1 \right) =\left(P^1,e^1 \right),
		\end{align*}
		and taking the imaginary part of the equation above, then we arrive at
		\begin{align*}
			\frac{1}{\mu}\|e^1\|^2=\mathrm{Im}\left( P^1,e^1\right) \leq\|P^1\|\|e^1\|,
		\end{align*}
		{\color{blue}for small $\tau$ and $K\geq2\Gamma(2-\alpha)C_u$},
		\begin{align*}
			{\color{blue}\|e^1\|\leq\mu\|P^1\|\leq C_u\Gamma(2-\alpha)\tau^{\alpha}\left(1+h^2+\tau^\alpha \right) \leq 2\Gamma(2-\alpha)C_u\left(\tau^\alpha+h^2 \right)\leq K\left(\tau^\alpha+h^2 \right)} ,
		\end{align*}
		where Lemma \ref{lemma2.5} and equation \eqref{eqution11} are used.\\
		\indent In order to estimate $|e^1|_2,$ we rewrite \eqref{eqution10} with
		$n = 1$ into $$\Delta_he_{j,k}^{1}=-\text{i}\frac{1}{\mu} e_{j,k}^{1}+P_{j,k}^{1},$$
		then taking the discrete L2 norm, we have
		\begin{align*}
			|e^1|_2\leq\frac{1}{\mu}\|e^1\|+\|P^1\|\leq2\|P^1\|\leq {\color{blue}2C_u\left(1+h^2+\tau^\alpha \right)}.
		\end{align*}
		Utilizing Lemma \ref{lemma3.2}, {\color{blue}for small $\tau$ and $h$}, we have
		\begin{align*}
			\|e^{1}\|_{\infty}\leq C\|e^{1}\|^{\frac12}(\|e^{1}\|+|e^{1}|_2)^{\frac12}\leq {\color{blue}C(K(\tau^\alpha+h^2))^{\frac{1}{2}}\left( K(\tau^\alpha+h^2)+2C_u\left(1+h^2+\tau^\alpha \right)\right) ^{\frac{1}{2}}\leq 1}.
		\end{align*}
		Hence, together with assumption \eqref{eqution8}, one has		
		\begin{align}\label{new}		
			\|U^{1}\|_\infty=\left\| u^1-\left( u^1-U^1\right) \right\|_{\infty} \leq\|u^1\|_\infty+\|e^{1}\|_\infty \leq {\color{blue}\max_{1\leq n\leq N}\|u^n\|_\infty+1}.
		\end{align}
		Therefore, Theorem \ref{theorem 3.2} holds for $n = 1$.\\
		\indent Now, we assume Theorem \ref{theorem 3.2} holds for $n \leq l-1$, $l\geq2$, i.e.
		\begin{align}\label{eqution12}		
			\|u^n-U^n\|\leq {\color{blue}K}(\tau t_n^{\alpha-1}+h^2),\quad\|U^n\|_\infty\leq {\color{blue}\max_{1\leq n\leq N}\|u^n\|_\infty+1},\quad n=1,2,\cdots,l-1.
		\end{align}	
		Next, we will show that Theorem \ref{theorem 3.2} holds for $n =l$,
		\begin{equation}\label{eqution18}
			\begin{aligned}
				\|R^{l}\|&=\left\|f(|u^{l-1}|^2)u^{l-1}-f(|U^{l-1}|^2)U^{l-1}\right\|\\
				&=\left\|\left(f(|u^{l-1}|^2)-f(|U^{l-1}|^2)\right)u^{l-1}+f(|U^{l-1}|^2)\left( u^{l-1}-U^{l-1}\right)\right\|\\
				&=\left\|f'(\xi_2)\left(|u^{l-1}|^2-|U^{l-1}|^2\right)u^{l-1}+f(|U^{l-1}|^2)e^{l-1}\right\|\\
				&\leq\left\|f'(\xi_2)\left(u^{n-1}\overline{e^{l-1}} +e^{l-1}\overline{U^{l-1}}\right)u^{l-1}\right\|+\left\|	f(|U^{l-1}|^2)e^{l-1}\right\|\\
				&\leq {\color{blue}C_1}\left\| e^{l-1}\right\|,
			\end{aligned}
		\end{equation}
		{\color{blue}where $f \in C^1(\mathbb{R}^+)$ is used,} and $\xi_2$ is a constant between $|u^{l-1}|^2$ and $|U^{l-1}|^2$, by \eqref{eqution8} and \eqref{eqution12}, it is easily seen that $\xi_2\leq C$. {\color{blue}Hence, the constant $C_1$ is independent of mesh sizes and the induction variable $l$.}\\		
		\indent Let $n = l$ in \eqref{eqution10}, then by taking the inner product of \eqref{eqution10} with $e^l_{j,k}$ and taking the imaginary part of the result, we arrive at
		\begin{align*}
			\mathrm{Re}(D_\tau^\alpha e^l,e^{l}) = -\mathrm{Im}\left( R^{l},e^{l}\right) +\mathrm{Im}\left( P^{l},e^{l}\right),
		\end{align*}
		then using the definition \eqref{eqution4} and Cauchy-Schwartz inequality, we have
		\begin{align*}
			\frac{1}{\mu}a_0\left\| e^l\right\| ^2-\frac{1}{\mu}\sum_{i=1}^{l-1}(a_{l-i-1}-a_{l-i})\left\| e^i\right\|\left\| e^l\right\|\leq\left\| R^l\right\|\left\| e^l\right\|+\left\| P^l\right\|\left\| e^l\right\|.
		\end{align*}
		Then it follows that
		\begin{align*}
			D^\alpha_\tau\left\| e^l\right\|\leq\left\| R^l\right\|+\left\| P^l\right\|\leq {\color{blue}C_1}\left\| e^{l-1}\right\|+{\color{blue}C_u} \left( \left\| r^l\right\|+h^2+\tau^\alpha l^{\alpha-1}\right).
		\end{align*}
		Applying the definition of operator $E_t^\alpha$ \eqref{eqution13} to both sides of the inequality above, then by Lemmas \ref{lemma2.4}, \ref{lemma2.6}, \ref{lemma2.2}, \ref{lemma2.3}, and \ref{lemma2.1}, one has
		\begin{align*}
			\left\| e^l\right\|&\leq {\color{blue}C_u}\sum_{i=1}^l\theta_{l-i} \left( \left\| r^i\right\|+h^2+\tau^\alpha i^{\alpha-1}\right)+{\color{blue}C_1}\sum_{i=1}^l\theta_{l-i}\left\| e^{i-1}\right\|\\
			&\leq {\color{blue}C_uC}\tau t^{\alpha-1}_l+{\color{blue}C_u}\sum_{i=1}^l\theta_{l-i}\left( h^2+\tau^\alpha i^{\alpha-1}\right)+{\color{blue}C_1}\sum_{i=1}^l\theta_{l-i}\left\| e^{i-1}\right\|\\
			&\leq {\color{blue}C_uC}\tau t^{\alpha-1}_l+{\color{blue}C_u\frac{T^\alpha}{\Gamma(1+\alpha)}}h^2+{\color{blue}C_2}\tau^{2\alpha}l^{2\alpha-1}+{\color{blue}C_1}\sum_{i=1}^{l-1}\theta_{l-i-1}\left\| e^{i}\right\|\\
			&\leq {\color{blue}C_uC}\tau t^{\alpha-1}_l+{\color{blue}C_u\frac{T^\alpha}{\Gamma(1+\alpha)}}h^2+{\color{blue}C_2}\tau t_l^{2\alpha-1}+{\color{blue}C_1\Gamma(2-\alpha)}\sum_{i=1}^{l-1}\tau^\alpha(l-i)^{\alpha-1}\left\| e^{i}\right\|\\
			&\leq {\color{blue}C_uC}\tau t^{\alpha-1}_l+{\color{blue}C_u\frac{T^\alpha}{\Gamma(1+\alpha)}}h^2+{\color{blue}C_2T^\alpha}\tau t_l^{\alpha-1}+{\color{blue}C_1\Gamma(2-\alpha)}\tau\sum_{i=1}^{l-1}t_{l-i}^{\alpha-1}\left\| e^{i}\right\|\\
			&= {\color{blue}(C_uC+C_2T^\alpha)}\tau t^{\alpha-1}_l+{\color{blue}C_u\frac{T^\alpha}{\Gamma(1+\alpha)}}h^2+{\color{blue}C_1\Gamma(2-\alpha)}\tau\sum_{i=1}^{l-1}t_{l-i}^{\alpha-1}\left\| e^{i}\right\|.
		\end{align*}
		By using Lemma \ref{lemma3.1}, 
	we obtain
		\begin{align}\label{errorin}
			\left\| e^l\right\|\leq {\color{blue}C_3}\left( \tau t^{\alpha-1}_l+h^2\right){\color{blue}\leq K(\tau t_n^{\alpha-1}+h^2)},
		\end{align}
		{\color{blue}where the constant $C_3$ is just dependent on $C, C_1, C_2$ and $C_u$, but independent of mesh sizes and the induction variable $l$, and
			\begin{equation} \label{keyc}
				K=\max\{2\Gamma(2-\alpha)C_u,C_3\}.
		\end{equation}}
		In order to estimate $|e^1|_2,$ we rewrite \eqref{eqution10} with $n = l$ into the following form
		\begin{align*}
			\Delta_he_{j,k}^{l}=-\text{i}D_\tau^\alpha e_{j,k}^l-R_{j,k}^{l}+P_{j,k}^{l},
		\end{align*}
		then taking the discrete L2 norm, we have
		\begin{align*}
			|e^{l}|_2& \leq\|D_\tau^\alpha e^l\|+\|R^{l}\|+\|P^{l}\| \\
			&\leq {\color{blue}\dfrac{2}{\tau^{\alpha}\Gamma(2-\alpha)}\max_{1\leq s\leq l}\|e^s\|}+{\color{blue}C_1}\|e^{l-1}\|+{\color{blue}C_u}\left( l^{-min\{\alpha+1,2-\alpha\}}+h^2+\tau^\alpha l^{\alpha-1}\right) \\
			&\leq {\color{blue}\dfrac{2}{\tau^{\alpha}\Gamma(2-\alpha)}K\left( \tau^\alpha +h^2\right)} +{\color{blue}C_1K}\left( \tau t_{l-1}^{\alpha-1}+h^2\right)+{\color{blue}C_u}\left( l^{-min\{\alpha+1,2-\alpha\}}+h^2+\tau^\alpha l^{\alpha-1}\right) \\
			&= {\color{blue}\dfrac{2K}{\Gamma(2-\alpha)}\left( 1+\tau^{-\alpha}h^2\right)}+{\color{blue}C_1K}\left( \tau^\alpha (l-1)^{\alpha-1}+h^2\right)+{\color{blue}C_u}\left( l^{-min\{\alpha+1,2-\alpha\}}+h^2+\tau^\alpha l^{\alpha-1}\right)\\
			&{\color{blue}\leq C_4(1+\tau^{-\alpha}h^2)},
		\end{align*}
		where Lemma \ref{lemma3.4} is used. Then using Lemma \ref{lemma3.2}, we have
		\begin{align*}
			||e^{l}||_\infty&\leq C||e^{l}||^{\frac12}\left(||e^{l}||+|e^{l}|_2\right)^{\frac12}\\
			&\leq {\color{blue}CK^{\frac12}}\left(\tau^\alpha l^{\alpha-1}+h^2 \right)^{\frac12}\left({\color{blue}K}\left(\tau^\alpha l^{\alpha-1}+h^2 \right) +{\color{blue}C_4(1+\tau^{-\alpha}h^2)}\right)^{\frac12} \\	
			&\leq {\color{blue}C_5}\left(\tau^\alpha l^{\alpha-1}+h^2 \right)^{\frac12}{\color{blue}\left(1+\tau^{-\alpha}h^2 \right)^{\frac12}}.
		\end{align*}
		Thus, for small $\tau$ and $h$, when {\color{blue}$h^2\leq \tau^\alpha l^{\alpha-1}$},
		\begin{align*}
			||e^{l}||_\infty&\leq {\color{blue}\sqrt{2}C_5}\left(\tau^\alpha l^{\alpha-1} \right)^{\frac12}{\color{blue}\left(1 +\tau^{-\alpha} \tau^\alpha l^{\alpha-1} \right)^{\frac12}}\\
			&= {\color{blue}\sqrt{2}C_5}\left(\tau^\alpha l^{\alpha-1} \right)^{\frac12}{\color{blue}\left(1 +l^{\alpha-1} \right)^{\frac12}}\\
			&={\color{blue}\sqrt{2}C_5}\left(\tau^{\alpha}  l^{\alpha-1}+\tau^{\alpha}  l^{2(\alpha-1)} \right)^{\frac12}\\
			&\leq {\color{blue}2C_5\tau^{\frac{\alpha}{2}}\leq 1}.
		\end{align*}
       
		On the other hand, when {\color{blue}$h^2> \tau^\alpha l^{\alpha-1}$}, using Lemma \ref{lemma3.3}, by \eqref{errorin} we obtain
		\begin{align*}
			||e^{l}||_\infty\leq Ch^{-1}||e^{l}||\leq {\color{blue}CKh^{-1}\left(\tau^\alpha l^{\alpha-1}+h^2\right)\leq 2CKh\leq 1}.
		\end{align*}
		Hence, for both the case {\color{blue}$h^2\leq \tau^\alpha l^{\alpha-1}$} and the case {\color{blue}$h^2> \tau^\alpha l^{\alpha-1}$}, we always have
		\begin{align*}
			\|U^{l}\|_\infty=\left\| u^l-\left( u^l-U^{l}\right) \right\| _\infty\leq\|u^{l}\|_\infty+\|e^{l}\|_\infty {\color{blue}\leq\max_{1\leq l\leq N}\|u^{l}\|_\infty+1}.
		\end{align*}
		Therefore, \eqref{eqution14} holds for $n = l$. This completes the proof of Theorem \ref{theorem 3.2}.
	\end{proof}
	\begin{remark}
			In previous papers, Jin et al. \cite{Jin} studied time-fractional nonlinear parabolic partial differential equations, and Hu et al. \cite{Hu}  developed a two-grid finite element method for solving the nonlinear time-fractional Schr\"{o}dinger equation. But they all conduct global error analysis, and the convergence order is $\alpha$,  which can be referred to  \cite[Theorem 7]{Jin} and \cite[Theorem 4.1]{Hu} for specific details. In contrast, our paper focuses on pointwise-in-time error estimate, which can achieve the convergence order of 1 when $t$ is away from $0$.
	\end{remark}

	\section{Numerical experiments} \label{section4}
	In this section, we provide several numerical experiments to verify the order of local convergence. And we utilize L2 norm error for estimation in these numerical examples.
	
	\indent Denote global error $E_g$ and local error $E_l$:
	\begin{align*}
		E_g:=\max_{1\le n\le N}\left\| u(t_n)-U^n\right\|,\quad E_l=\left\| u(t_N)-U^N\right\|.
	\end{align*}
	\begin{example}\label{example4.1}
		Consider the following two-dimensional nonlinear time-fractional Schr\"{o}dinger equation
		\begin{align}\label{eqnum}
			{\rm{i}}D_t^\alpha u+\Delta u+|u|^2u=f,
		\end{align}
		with the exact solution
		\begin{align*}
			u(x,y,t)=(t^\alpha-1)(1+{\rm{i}})\sin(\pi x)\sin(\pi y),
		\end{align*}
		which exhibits typical initial weak singularity, and the force term $f$ can by computed by \eqref{eqnum}.
	\end{example}
	Taking $L = 1$, $T = 1$, and we take $M = \lceil N^\frac{1}{2}\rceil$ so that the temporal error dominates the spatial error. Table \ref{table4.1} shows the local errors and convergence orders with  $\alpha=0.3,0.5,0.7$, one can see that the results conform with Theorem \ref{theorem 3.2}.
	In addition, we also estimate the global error, the convergence order $\alpha$ is obtained in Table \ref{table4.1}. Moreover, we choose different grid ratios $\tau/h=0.1,10$ to test the unconditional convergence of our scheme. Those results from Table \ref{table4.2} confirm the convergence without any grid ratio given in Theorem \ref{theorem 3.2}.
	\begin{table}[H]
		\centering
		\caption{Local and global $L^2$ errors and rates of convergence for Example \ref{example4.1}. }
		\begin{tabular}{ccccccc}
			\toprule
			&     N    & 512 & 1024 & 2048 & 4096& 8192\\
			\midrule
			&$\alpha=0.3$  & 9.0586e-06 &4.5730e-06 &2.1972e-06 & 1.1167e-06 & 5.4866e-07\\
			&             &  & 0.9861 & 1.0575 & 0.9765 & 1.0252\\
			$E_l$ &$\alpha=0.5$  & 2.0332e-05 &1.0409e-05 & 5.0306e-06 & 2.5843e-06 & 1.2769e-06\\
			&          &  & 0.9660 & 1.0490 & 0.9610 & 1.0171\\
			&$\alpha=0.7$  &3.5230e-05 &1.8061e-05 & 8.7224e-06 & 4.4832e-06 & 2.2135e-06\\
			&       	  &  & 0.9640 & 1.0501 & 0.9602  & 1.0182\\
			&           &  &  & &  &\\
			
			&$\alpha=0.3$  &1.7320e-02  &1.4312e-02 &1.1843e-02 & 9.8220e-03 &8.1364e-03\\
			&	         & &0.2752 & 0.2732 &0.2700 &0.2716\\
			$E_g$ &$\alpha=0.5$  & 6.5287e-03 & 4.6379e-03 &3.3033e-03 & 2.3515e-03 & 1.6694e-03\\
			&	          & &0.4933 &0.4896 & 0.4903 & 0.4942\\
			&$\alpha=0.7$  &1.7573e-03 &1.0318e-03 & 6.2977e-04 & 3.8671e-04 & 2.3786e-04\\
			&       	  &  & 0.7682 & 0.7122 & 0.7036 & 0.7012\\		
			\bottomrule
		\end{tabular} \label{table4.1}
	\end{table}

\begin{table}[H]
		\centering
		\caption{Local errors under different grid ratios for Example \ref{example4.1} with $\alpha=0.5$.}
		\begin{tabular}{c|c|c}
			\hline
			grid ratio&step size&error\\
			\hline
			\multirow{2}{*}{$\frac{\tau}{h}=0.1$} & $\tau=0.01,h=0.1$ & 1.1053e-04\\
			\cline{2-3}
			& $\tau=0.005,h=0.05$ & 2.9888e-05 \\
			\hline
			\multirow{2}{*}{$\frac{\tau}{h}=10$} & $\tau=0.1,h=0.01$ & 2.3201e-04\\
			\cline{2-3}
			& $\tau=0.05,h=0.005$ & 9.2144e-05 \\
			\hline
		\end{tabular} \label{table4.2}
\end{table}

	\begin{example}\label{example4.2}
		Consider the following two-dimensional nonlinear time fractional Schr\"{o}dinger equation
		\begin{align*}
			{\rm{i}}D_t^\alpha u+\Delta u+|u|^2u=0,
		\end{align*}
		with the initial condition
		\begin{align*}
			u(x,y,0)=sin(\pi x)sin(\pi y).
		\end{align*}
	\end{example}
	Due to the unknown exact solution of this Example \ref{example4.2}, the two-mesh method given in \cite{Farrell} is used to estimate the errors and convergence rates. The local error $e_L$ of the two-mesh are defined by
	\begin{align*}
		e_L=\left\| u^N-w^N\right\|,
	\end{align*}
	where $w^k$ is computed on the second temporal mesh $t_k = T (k/2N)$ with $0 \leq k \leq 2N$. We choose $M = 50$ to test
	the temporal convergence rate thus the temporal errors are dominated when $N$ is not too big. Those results from
	Table \ref{table4.3} confirm again the theoretical prediction given in Theorem \ref{theorem 3.2}.
	
	\begin{table}[H]
		\centering
		\caption{Local $L^2$ errors and rates of convergence for Example \ref{example4.2}. }
		\begin{tabular}{c|c|c|c|c|c|c}
			\hline
		
			\multirow{2}{*}{$N$} & \multicolumn{2}{c|}{$\alpha=0.3$} &
			\multicolumn{2}{c|}{$\alpha=0.5$} & \multicolumn{2}{c}{$\alpha=0.7$}\\
			\cline{2-7}
			{} & $e_{L}$ & $Rate$ & $e_{L}$ & $Rate$ & $e_{L}$ & $Rate$  \\
			\hline
			64 &4.3992e-05 &    &5.4087e-05 &    &4.5462e-05 & \\
			\hline
			128 &2.1787e-05 & 1.0138   &2.6774e-05 &  1.0144     &2.2562e-05 & 1.0108\\
			\hline
			256 &1.0843e-05 &1.0068 &1.3342e-05 &1.0048 &1.1262e-05 &1.0024 \\
			\hline
			512 &5.4089e-06 &1.0033 &6.6718e-06 &0.9999 &5.6242e-06 &1.0017 \\
			\hline
			1024 &2.7016e-06 &1.0015 &3.3396e-06 &0.9984 &2.8058e-06 &1.0032 \\
			\hline
		\end{tabular} \label{table4.3}
	\end{table}

\begin{example}\label{example4.3}
		Consider the following two-dimensional nonlinear time fractional Schr\"{o}dinger equation
		\begin{align*}
			{\rm{i}}D_t^\alpha u+\Delta u+|u|^2u=0,
		\end{align*}
		with the non-smooth initial condition
		\begin{align*}
			u(x,y,0)=\begin{cases}xsin(\pi y),&0\leq x\leq 0.5,0\leq y\leq 1,\\ 0,&0.5<x\leq 1,0\leq y\leq 1.\end{cases}
		\end{align*}
	\end{example}
	Similar to Example \ref{example4.2}, we utilize the two-mesh method to estimate the errors and convergence rates. We choose $M=50$ to test the temporal convergence rate with the non-smooth initial condition. The data presented in
	Table \ref{table4.4} is consistent with the theoretical prediction stated in Theorem \ref{theorem 3.2}.
	
	\begin{table}[H]
		\centering
		\caption{Local $L^2$ errors and rates of convergence for Example \ref{example4.3}. }
		\begin{tabular}{c|c|c|c|c|c|c}
			\hline
			\multirow{2}{*}{$N$} & \multicolumn{2}{c|}{$\alpha=0.3$} &
			\multicolumn{2}{c|}{$\alpha=0.5$} & \multicolumn{2}{c}{$\alpha=0.7$}\\
			\cline{2-7}
			{} & $e_{L}$ & $Rate$ & $e_{L}$ & $Rate$ & $e_{L}$ & $Rate$  \\
			\hline
			64 &1.0507e-05 &    &1.2904e-05 &    &1.0665e-05 & \\
			\hline
			128 &5.1978e-06 & 1.0154   &6.3725e-06 &  1.0179     &5.2610e-06 & 1.0195\\
			\hline
			256 &2.5851e-06 &1.0077 &3.1667e-06 &1.0089 &2.6121e-06 &1.0101 \\
			\hline
			512 &1.2890e-06 &1.0039 &1.5787e-06 &1.0042 &1.3009e-06 &1.0057 \\
			\hline
			1024 &6.4355e-07 &1.0022 &7.8826e-07 &1.0020 &6.4937e-07 &1.0024 \\
			\hline
		\end{tabular} \label{table4.4}
	\end{table}
	\section{Conclusion}
	In this paper, we have established a local error estimate of the L1 scheme for the two-dimensional nonlinear time fractional Schr\"{o}dinger equation with non-smooth solutions. Stability and pointwise-in-time convergence of the fully discrete scheme are given. It is proven that the proposed scheme achieves a convergence rate of $O(\tau t_n^{\alpha-1}+h^2)$ without imposing any time-step restrictions that depend on the spatial mesh size.
	
	\section*{Acknowledgements}
	The research is supported in part by Natural Science Foundation of Shandong Province under Grant ZR2023MA077, the National Natural Science Foundation of China (Nos.  11801026 and 12171141), and Fundamental Research Funds for the Central Universities (No. 202264006).
	
	\section*{Declaration of competing interest}
	The authors declare no competing interest.
	
	
	
	\bibliographystyle{plain}
	\bibliography{NonSchrL1v1ref}
	
	
	%
	%
	%
\end{document}